\documentclass{amsart}

\usepackage[utf8]{inputenc}

\usepackage{hyperref}

\newcommand{\area}{{\rm Area}}

\newcommand{\vol}{{\rm Vol}}

\newcommand{\DN}{{\mathcal D}}
\newcommand{\Tr}{\operatorname{Tr}}
\newcommand{\oo}{{\mathcal{O}}}

\usepackage{amsthm,amsmath,amsfonts, amssymb}

\usepackage{graphicx}

\usepackage[usenames,dvipsnames]{color}

\makeatletter
\def\th@alexnormal{%
\let\thm@indent\noindent 
\thm@headfont{\bfseries}
\normalfont
}
\makeatother

\makeatletter
\def\th@alexit{%
\let\thm@indent\noindent 
\thm@headfont{\bfseries}
\normalfont
\fontshape{it}
\selectfont
}
\makeatother

\theoremstyle{alexit}

\newtheorem{theorem}[equation]{Theorem}
\newtheorem{proposition}[equation]{Proposition}
\newtheorem{lemma}[equation]{Lemma}
\newtheorem{cor}[equation]{Corollary}
\newtheorem{corollary}[equation]{Corollary}

\theoremstyle{remark}
\newtheorem{remark}[equation]{Remark} 
\theoremstyle{definition}

\newtheorem{example}[equation]{Example}
\numberwithin{equation}{subsection}

\newtheorem{open}{Open Problem}

\begin{document}
\author{Alexandre Girouard}
\address{Département de mathématiques et de statistique, 
Pavillon Alexandre-Vachon, Université Laval,
Québec, QC, G1V 0A6, Canada}
\email{alexandre.girouard@mat.ulaval.ca}
\thanks{Partially supported by FRQNT}
\author{Iosif Polterovich}
\address{D\'epartement de math\'ematiques et de
statistique, Universit\'e de Montr\'eal, C. P. 6128,
Succ. Centre-ville, Montr\'eal, QC,  H3C 3J7,  Canada}
\email{iossif@dms.umontreal.ca}
\thanks{Partially supported by NSERC, FRQNT and Canada Research Chairs program}
\title{Spectral geometry of the Steklov problem}


\begin{abstract}
  The Steklov problem is an eigenvalue problem with the spectral
  parameter in the boundary conditions, which has various
  applications. Its spectrum coincides with that of the
  Dirichlet-to-Neumann operator. Over the past years, there has been a
  growing interest in the Steklov problem from the viewpoint of spectral
  geometry. While this problem shares some common properties with its
  more familiar Dirichlet and Neumann cousins, its eigenvalues and
  eigenfunctions have a number of distinctive geometric features, which
  makes the subject especially appealing. In this survey we discuss some
  recent advances and open questions, particularly in the study of
  spectral asymptotics, spectral invariants, eigenvalue estimates, and
  nodal geometry.
\end{abstract}

\maketitle


\section{Introduction}
\subsection{The Steklov problem}
Let $\Omega$ be a  compact Riemannian manifold of dimension $n\ge
2$ with (possibly non-smooth) boundary $M=\partial\Omega$.
The
\emph{Steklov problem} on $\Omega$ is
\begin{equation}
\label{stek}
\begin{cases}
  \Delta u=0& \mbox{ in } \Omega,\\
  \partial_\nu u=\sigma \,u& \mbox{ on }M.
\end{cases}
\end{equation}
where $\Delta$ is the Laplace-Beltrami operator acting on functions on
$\Omega$, and $\partial_\nu$ is the outward normal derivative along the
boundary $M$. This problem was introduced by the Russian mathematician
V.A. Steklov at the turn of the 20th century (see \cite{KKKNPPS} for a historical discussion).
It is well known that the spectrum of the Steklov problem is discrete as long as the trace
operator $H^1(\Omega)\rightarrow L^2(\partial\Omega)$ is compact (see~\cite{arendtmazzeo}).  In this case, the eigenvalues 
 form a sequence
$0=\sigma_0<\sigma_1\leq\sigma_2\leq\cdots\nearrow\infty$.
This is true under some mild regularity assumptions, for instance if
$\Omega$ has Lipschitz boundary (see~\cite[Theorem 6.2]{necas}).

The present paper focuses on the geometric properties of Steklov eigenvalues and eigenfunctions.
A lot of progress in this area has been made in the last few years,
and some fascinating open problems have emerged.  
We will start by explaining the motivation to study the Steklov spectrum.
In particular, we will emphasize the differences between this eigenvalue problem and its
Dirichlet and Neumann counterparts.


\subsection{Motivation}
The Steklov eigenvalues can be interpreted as the eigenvalues of the
{\it Dirichlet-to-Neumann operator}  $\DN:H^{1/2}(M) \rightarrow
H^{-1/2}(M)$ which maps a function  $f\in H^{1/2}(M)$ to  $\DN
f=\partial_\nu(Hf)$, where $Hf$ is the harmonic
extension of $f$ to $\Omega$. The study of the Dirichlet-to-Neumann
operator (also known as the voltage-to-current map)  is essential for
applications to electrical impedance tomography, which is used in
medical and geophysical imaging (see~\cite{Uhlmann} for a recent
survey).

A rather striking feature of the asymptotic
distribution of Steklov eigenvalues is its unusually (compared to the Dirichlet and Neumann cases) high sensitivity
to the regularity of the boundary.  On one hand, if the boundary of a
domain is smooth, the corresponding Dirichlet-to-Neumann operator is
pseudodifferential and elliptic of order  one (see ~\cite{Taylor}).  As a
result, one can show, for instance, that a surprisingly  sharp 
asymptotic formula for Steklov eigenvalues ~\eqref{formula:GPPSmain} holds for smooth surfaces. 
However, this estimate already
fails for polygons (see section~\ref{sing}). It is in fact
likely that for domains which are not $C^\infty$-smooth but only of
class $C^k$ for some $k \ge 1$, the rate of decay of the remainder
in eigenvalue asymptotics depends on $k$. To our knowledge, for domains with Lipschitz
boundaries, even one-term spectral asymptotics have not yet been
proved. A summary of the available  results is presented
in~\cite{Agranovich} (see also~\cite{AgranovichAmosov}).

One of the oldest topics in spectral geometry is shape
optimization. Here again, the Steklov spectrum holds some
surprises. For instance, the classical result of Faber--Krahn for the
first Dirichlet eigenvalue $\lambda_1(\Omega)$ states that among
Euclidean domains with fixed measure, $\lambda_1$ is minimized by a
ball. Similarly, the Szeg\H{o}--Weinberger inequality states that the first nonzero
Neumann eigenvalue $\mu_1(\Omega)$ is maximized by a ball. In both
cases, no topological assumptions are made. The analogous result
for Steklov eigenvalues is Weinstock's inequality, which states that
among planar domains with fixed perimeter, $\sigma_1$ is maximized by
a disk provided that $\Omega$ is simply--connected. In contrast with
the Dirichlet and Neumann case, this assumption cannot be removed. Indeed
the result fails for appropriate annuli (see
section~\ref{isoperim}).  Moreover, maximization of  the first Steklov eigenvalue among all planar domains of given perimeter is an open problem.  
At the same time, it is known that for simply--connected planar domains,  the $k$-th normalized Steklov eigenvalue is maximized in
the limit by a disjoint union of $k$ identical disks  for
any $k\ge 1$ \cite{gp}. Once again, for  the Dirichlet and Neumann eigenvalues the situation is quite different: the extremal domains for $k\ge 3$ are  
known only at the level of experimental numerics, and, with a few exceptions, are expected to have rather complicated geometries.

Probably the most well--known question in spectral geometry is ``Can one hear the shape of a drum?'', or whether there
exist domains or manifolds that are isospectral but not isometric. 
Apart from some easy examples discussed in section~\ref{isosp}, no
examples of Steklov isospectral non-isometric manifolds are
presently known. Their construction appears to be even trickier than for the
Dirichlet or Neumann problems. In particular, it is not known whether
there exist Steklov isospectral Euclidean domains which are not isometric. Note
that the standard transplantation techniques (see \cite{Berard,
  Buser1986, BCDS}) are not applicable for the Steklov problem, as it
is not clear how to “reflect” Steklov eigenfunctions across the
boundary.

New challenges also arise in the study of the nodal domains and the nodal sets
of Steklov eigenfunctions. One of the problems is to understand
whether the nodal lines of Steklov eigenfunctions are dense at the
``wave-length scale'', which is a basic property of the zeros of
Laplace eigenfunctions. Another interesting question is the nodal
count for the Dirichlet-to-Neumann eigenfunctions.  We touch upon
these topics in section~\ref{nodal}.

Let us conclude this discussion  by mentioning that the Steklov
problem is often considered in the more general form 
\begin{equation}
  \label{stekgen}
  \partial_\nu u = \sigma \rho u,
\end{equation}
where $\rho \in L^{\infty}(\partial \Omega)$ is a non-negative weight function on the boundary. If $\Omega$ is two-dimensional, the Steklov eigenvalues can be thought of 
as the squares of the natural frequencies of a vibrating free membrane with
its mass concentrated along its boundary with density $\rho$ (see~\cite{LambPro}).  A
special case of the Steklov problem with the boundary condition
\eqref{stekgen} is the sloshing problem, which describes the
oscillations of fluid in a container.
In this case, $\rho \equiv 1$ on
the free surface of the fluid and $\rho \equiv 0$ on the walls of the
container. There is an extensive literature on the properties of
sloshing eigenvalues and eigenfunctions,
see~\cite{foxkut,BanKulPoltSiu,KozKuz} and references
therein.

Since the present survey is directed towards geometric questions, in
order to simplify the analysis and presentation we focus  on the pure
Steklov problem with $\rho\equiv 1$.

\subsection{Computational examples}\label{Section:Examples}
The Steklov spectrum can be explicitly computed in  a few cases.
Below we discuss the Steklov
eigenvalues and eigenfunctions of cylinders and balls using
separation of variables.

\begin{example}
  The Steklov eigenvalues of a unit disk are 
  $$0,1,1,2,2,\dots,k,k,\dots.$$
  The corresponding eigenfunctions in polar coordinates $(r, \phi)$ 
  are given by $$1, r\sin \phi, r\cos\phi,\dots, r\sin k\phi, r\cos k\phi,\dots.$$
\end{example} 
\begin{example}
\label{balls}
  The Steklov eigenspaces on the ball $B(0,R)\subset\mathbb{R}^n$ are
  the restrictions of the spaces $H_k^n$ of homogeneous harmonic
  polynomials of degree $l\in\mathbb{N}$ on $\mathbb{R}^n$. The
  corresponding eigenvalue is $\sigma=k/R$ with multiplicity 
  $$\dim H_k^n={n+k-1\choose n-1}-{n+k-3\choose n-1}.$$
  This is of course a generalization of the previous example.
\end{example}
\begin{example}\label{cyl}
  This example is taken from~\cite{ceg2}. Let $\Sigma$ be a
  compact  Riemannian manifold without
  boundary. Let 
  $$0=\lambda_1<\lambda_2\leq\lambda_3\cdots\nearrow\infty$$ 
  be the spectrum of the Laplace-Beltrami operator $\Delta_\Sigma$ on
  $\Sigma$, and let $(u_k)$ be an orthonormal basis of $L^2(\Sigma)$
  such that
  $$\Delta_{\Sigma} u_k=\lambda_ku_k.$$
  Given any $L>0$, consider the cylinder $\Omega=[-L,L]\times\Sigma \subset \mathbb{R}\times\Sigma$.
  Its Steklov spectrum is given by 
  \begin{gather*}
    0, 1/L,\ 
    \sqrt{\lambda_k}\tanh (\sqrt{\lambda_k}L)<
    \sqrt{\lambda_k}\coth (\sqrt{\lambda_k}L).
  \end{gather*}
  and the corresponding eigenfunctions are
  \begin{gather*}
    1,\ t,\
    \cosh(\sqrt{\lambda_k}t)f_k(x),\
    \sinh(\sqrt{\lambda_k}t)f_k(x).
  \end{gather*}
\end{example}

In sections \ref{section:spectrumsquare} and \ref{isoperim}  we will discuss two more
computational examples: the  Steklov eigenvalues of a
square and of annuli.

\subsection{Plan of the paper}
The paper is organized as follows. 
In section~\ref{Section:Asymptotics} we survey results on the asymptotics and invariants of the Steklov
spectrum on smooth Riemannian manifolds.  In section \ref{sing} we discuss asymptotics of Steklov eigenvalues on polygons, which turns out to be quite different from the case of smooth planar domains.
Section~\ref{section:geometricinequalities} is concerned
with geometric inequalities.
In section~\ref{section:IsospectralityRigidity} we discuss Steklov isospectrality and
spectral rigidity.  Finally,  section~\ref{section:NodalMultiplicity} deals with the nodal geometry of
Steklov eigenfunctions and the multiplicity bounds for Steklov
eigenvalues.  

\section{Asymptotics and  invariants of the Steklov
  spectrum}\label{Section:Asymptotics}
\subsection{Eigenvalue asymptotics}
As above, let $n\geq 2$ be the dimension of the manifold $\Omega$, so that the dimension of the
boundary $M=\partial\Omega$ is $n-1$.
As was mentioned in the introduction, the Steklov eigenvalues of a compact manifold $\Omega$ with boundary
$M=\partial\Omega$ are the eigenvalues of the
Dirichlet-to-Neumann map.  It  is a first order elliptic pseudodifferential operator
which has the same principal symbol as the square root of the
Laplace-Beltrami operator on $M$. Therefore, applying standard
results of H\"ormander~\cite{hormander}  we obtain the following
Weyl's law for Steklov eigenvalues: 
$$
\#(\sigma_j < \sigma)=\frac{\vol(\mathbb{B}^{n-1})\,\vol (M)}{(2\pi)^{n-1}} \sigma^{n-1} + \oo(\sigma^{n-2}),
$$
where $\mathbb{B}^{n-1}$ is a unit ball in $\mathbb{R}^{n-1}$.
This formula can be rewritten
\begin{equation}
\label{Weylaw}
\sigma_j= 2\pi \left(\frac{j}{\vol(\mathbb{B}^{n-1})\, \vol(M)}\right)^\frac{1}{n-1} + \oo(1).
\end{equation}
In two dimensions, a much more precise asymptotic formula was proved in~\cite{GPPS}.  Given a finite sequence $C=\{\alpha_1,\cdots,\alpha_k\}$ of positive numbers,
consider the following union of multisets (i.e. sets with multiplicities):
$\{0,..\dots,0\}\cup \alpha_1 \mathbb{N}\cup \alpha_1 \mathbb{N}\cup
\alpha_2\mathbb N\cup\alpha_2\mathbb N\cup \dots
\cup\alpha_k\mathbb{N}\cup \alpha_k\mathbb{N}$, where the first multiset
contains $k$ zeros and $\alpha\mathbb{N}=\{\alpha,2\alpha,3\alpha,\dots,n\alpha,\dots\}$.
We rearrange the elements of this multiset into a monotone increasing sequence $S(C)$.
For example,
$S(\{1\})=\{0,1,1,2,2,3,3,\cdots\}$ and $S(\{1, \pi\})=\{0,0,1,1,2,2,3,3,\pi,\pi,4,4,5,5,6,6,2\pi,2\pi,7,7,\cdots\}$.
The following sharp spectral estimate was proved in~\cite{GPPS}.
\begin{theorem}
\label{main:GPPS}
  Let $\Omega$ be a smooth compact Riemannian surface with boundary $M$.
  Let $M_1,\cdots,M_k$ be the connected components of the
  boundary $M=\partial\Omega$, with lengths $\ell(M_i), 1\leq i\leq
  k$. 
  Set $R=\left\{\frac{2\pi}{\ell(M_1)},\cdots,\frac{2\pi}{\ell(M_k)}\right\}$. Then
  \begin{gather}\label{formula:GPPSmain}
    \sigma_{j}=S(R)_j+\oo(j^{-\infty}),
  \end{gather}
where $\mathcal O(j^{-\infty})$ means that the error term decays
faster than any negative power of~$j$.
\end{theorem}
In particular, for simply--connected surfaces we recover the following
result proved earlier by Rozenblyum and Guillemin--Melrose
(see~\cite{rozenbljumEnglish,edward}): 
\begin{equation}
\label{refined}
\sigma_{2j}=\sigma_{2j-1} + \oo(j^{-\infty})=\frac{2\pi}{\ell(M)}j+\oo(j^{-\infty}) .
\end{equation}
The idea of the proof of Theorem \ref{main:GPPS} is as follows.  For each boundary component $M_i$, $i=1,\dots, k$,  we cut off a  ``collar'' neighbourhood of the boundary and smoothly  glue a cap onto it. In this way,  one obtains $k$ simply--connected surfaces, whose boundaries are precisely $M_1, \dots, M_k$, and the Riemannian metric in the neighbourhood of each $M_i$, $i=1,\dots k$,  coincides with the metric on $\Omega$. Denote by $\Omega^*$ the union of these simply--connected surfaces.
Using an explicit formula for the full symbol of the
Dirichlet-to-Neumann operator \cite{LeeUhlmann} we notice that the
Dirichlet-to-Neumann operators associated with $\Omega$ and $\Omega^*$
differ by a smoothing operator, that is, by a pseudodifferential operator with a smooth integral kernel; such operators are bounded as maps between any two Sobolev spaces $H^s(M)$ and $H^t(M)$, $s,t \in \mathbb{R}$.
Applying pseudodifferential techniques,  we deduce that the
corresponding Steklov eigenvalues of $\sigma_j(\Omega)$ and
$\sigma_j(\Omega^*)$ differ by $\oo(j^{-\infty})$. Note that a similar
idea was used in \cite{HislopLutzer}. Now, in order to study the
asymptotics of the Steklov spectrum of $\Omega^*$,  we map each of its
connected components to a disk by a conformal transformation and apply
the approach  of Rozenblyum-Guillemin-Melrose which is also based on
pseudodifferential calculus. 
\subsection{Spectral invariants} 
\label{specinv}
The following result is an immediate corollary of Weyl's law~(\ref{Weylaw}).
\begin{corollary}
\label{cor1}
  The Steklov spectrum determines the dimension of the manifold and
  the volume of its boundary.
\end{corollary}
More refined information can be extracted from the Steklov spectrum of surfaces. 
\begin{theorem}\label{xx}
  The Steklov spectrum determines the number $k$ and the
    lengths $\ell_1\geq \ell_2\geq\cdots\geq\ell_k$ of
    the boundary components of a smooth compact Riemannian
  surface. Moreover, if  $\{\sigma_j\}$ is  the monotone increasing
  sequence of Steklov eigenvalues, then:
  
  $$\ell_1=\frac{2\pi}{\limsup_{j\rightarrow\infty}(\sigma_{j+1}-\sigma_j)}.$$
\end{theorem}
This result is proved in~\cite{GPPS} by a combination of
Theorem~\ref{main:GPPS} and certain number-theoretic arguments
involving the Dirichlet theorem on simultaneous approximation of
irrational numbers.

As was shown in~\cite{GPPS}, a direct generalization of
  Theorem \ref{xx} to higher dimensions is false.  Indeed, consider
  four flat rectangular tori: $T_{1,1}=\mathbb{R}^2/\mathbb{Z}^2$,
  $T_{2,1} =\mathbb{R}/2\mathbb{Z}\times\mathbb{R}/\mathbb{Z}$, 
  $T_{2,2}~=~\mathbb{R}^2/(2\mathbb{Z})^2$ and
  $T_{\sqrt{2},\sqrt{2}}=\mathbb{R}^2/(\sqrt{2}\mathbb{Z})^2$. It was shown in~\cite{DoyleRossetti,Parzanchevski} 
  that the disjoint union $\mathcal{T}=T_{1,1}\sqcup T_{1,1}\sqcup T_{2,2}$ is Laplace--Beltrami isospectral to the disjoint union $\mathcal{T}'=T_{2,1}\sqcup T_{2,1}\sqcup 
  T_{\sqrt{2},\sqrt{2}}$.  It follows from Example~\ref{cyl} that for any
  $L>0$, the two disjoint unions of cylinders $\Omega_1=[0,L]\times
  \mathcal{T}$  and $\Omega_2=[0,L] \times \mathcal{T}'$ are  Steklov
  isospectral. At the same time, 
  $\Omega_1$ has four boundary components of area $1$ and two  boundary components of area $4$, while $\Omega_2$ has  six 
  boundary components of  area $2$.  Therefore, the collection of
  areas of boundary components cannot be determined from the Steklov
  spectrum.
Still, the following question can be asked:
\begin{open}
  Is the number of boundary components  of a manifold of dimension $\ge 3$ a Steklov spectral invariant?
\end{open}

Further spectral invariants can be deduced using the heat trace of the
Dirichlet-to-Neumann operator $\DN$. 
By the results of~\cite{DuistermaatGuillemin,Agranovich2,GrubbSeeley},
the heat trace admits an asymptotic expansion
\begin{equation}
\label{heatexp}
\sum_{i=0}^{\infty}e^{-t\sigma_i}=\Tr  e^{-t\DN}=\int_M e^{-t\DN}(x,x)\ dx\sim\sum_{k=0}^{\infty}a_kt^{-n+1+k}+\sum_{l=1}^{\infty}b_lt^l\log t.
\end{equation}
The coefficients $a_k$ and $b_l$ are called the \emph{Steklov heat
invariants}, and it follows from (\ref{heatexp}) that they are
determined by the Steklov spectrum. The invariants
$a_0,\ldots,a_{n-1}$, as well as $b_l$ for all $l$, are local, in the
sense that they are integrals over $M$ of corresponding functions
$a_k(x)$ and $b_l(x)$ which may be computed directly from the symbol
of the Dirichlet-to-Neumann operator $\DN$.  The coefficients $a_k$
are not local for $k\geq n$~\cite{Gilkey,GilkeyGrubb} and hence are
significantly more difficult to study.

In~\cite{PoltSher},  explicit  expressions for the Steklov heat invariants
$a_0$, $a_1$ and $a_2$  for manifolds of dimensions three or
  higher were given in terms of the scalar curvatures of $M$ and
$\Omega$, as well as the mean curvature and the second order mean
curvature of $M$ (for further results in this direction, see~\cite{liu}).
For example, the formula for $a_1$  yields the
following corollary: 
\begin{corollary}
\label{cor2}
  Let $\dim \Omega \ge 3$. Then the  integral of the mean curvature
  over $\partial \Omega=M$ (i.e. the {\it total mean curvature} of
  $M$) is an invariant of the Steklov spectrum.
\end{corollary}
The Steklov heat invariants will be revisited in section~\ref{isosp}.
\section{Spectral asymptotics on polygons}
\label{sing}
The spectral asymptotics given by formula  \eqref{Weylaw} and Theorem \ref{main:GPPS} are obtained using pseudodifferential techniques which are valid
only for manifolds with smooth boundaries. In the presence of singularities, the study of the asymptotic distribution of Steklov eigenvalues is more difficult,
and the known remainder estimates are significantly weaker
(see~\cite{Agranovich} and references therein). Moreover, Theorem
\ref{main:GPPS} fails even for planar domains with corners.  This can
be seen from the explicit computation of the spectrum for the
simplest nonsmooth domain: the  square.
\begin{table}\label{table:eigenvalues}
\begin{tabular}{llll}
  Eigenspace basis&Conditions on $\alpha$&
  Eigenvalues&Asymptotic    behaviour\\[.2cm]
  $\cos(\alpha x)\cosh(\alpha y)$
  &&\\
  $\cos(\alpha y)\cosh(\alpha x)$
  &$\tan(\alpha )=-\tanh(\alpha)$
  &$\alpha\tanh(\alpha)$&$\frac{3\pi}{4}+\pi j+\oo(j^{-\infty})$\\[.2cm]
  \hline
  $\sin(\alpha x)\cosh(\alpha y)$
  &\\
  $\sin(\alpha y)\cosh(\alpha x)$
  &$\tan(\alpha )=\coth(\alpha)$
  &  $\alpha\tanh(\alpha)$&$\frac{\pi}{4}+\pi j+\oo(j^{-\infty})$\\[.2cm]
  \hline
  $\cos(\alpha x)\sinh(\alpha y)$
  &\\
  $\cos(\alpha y)\sinh(\alpha x)$
  &$\tan(\alpha )=-\coth(\alpha)$&$\alpha\coth(\alpha)$&$\frac{3\pi}{4}+\pi j+\oo(j^{-\infty})$\\[.2cm]
  \hline
  $\sin(\alpha x)\sinh(\alpha y)$
  &\\
  $\sin(\alpha y)\sinh(\alpha x)$
  &$\tan(\alpha)=\tanh(\alpha)$&$\alpha\coth(\alpha)$&$\frac{\pi}{4}+\pi j+\oo(j^{-\infty})$\\[.2cm]
  \hline
  $xy$&&1
\end{tabular}

\bigskip

\caption{Eigenfunctions obtained by separation of variables on the
  square $(-1,1)\times (-1,1)$.}
\label{SquareSepVar} 
\end{table}

\subsection{Spectral asymptotics on the square}
\label{section:spectrumsquare}
  The Steklov spectrum of the square $\Omega=(-1,1)\times (-1,1)$ is
  described as follows.
  For each positive root $\alpha$ of the following  equations:
  \begin{align*}
    \tan(\alpha )+\tanh(\alpha )=0,\quad 
    \tan(\alpha )-\coth(\alpha )=0,\\
    \tan(\alpha )+\coth(\alpha )=0,\quad
    \tan(\alpha)-\tanh(\alpha)=0
  \end{align*}
  the number $|\alpha\tan(\alpha)|$ is a Steklov eigenvalue  of multiplicity two (see
  Table~\ref{table:eigenvalues} and Figure~\ref{figure:square}). 
  The function $f(x,y)=xy$ is also an eigenfunction, with a simple
  eigenvalue $\sigma_3=1$. Starting from $\sigma_4$,  the normalized eigenvalues
  are clustered in groups of $4$ around the odd multiples of $2\pi$:
  \begin{gather*}
    \sigma_{4j+l}L=(2j+1)2\pi+\oo(j^{-\infty}),\qquad\mbox{ for }l=0,1,2,3.
  \end{gather*}
  This is compatible with Weyl's law since for $k=4j+l$ it follows that
  $$\sigma_kL=\left(\frac{k-l}{2}+1\right)2\pi+\oo(j^{-\infty})=\pi k+\oo(1).$$
  Nevertheless, the refined asymptotics~(\ref{refined}) does not
  hold.

Let us discuss the spectrum of a square in more detail.
Separation of variables  quickly leads to the 8 families of Steklov eigenfunctions
presented in Table~\ref{SquareSepVar} plus an ``exceptional'' eigenfunction $f(x,y)=xy$.
\begin{figure}
  \centering
  \includegraphics[width=10cm]{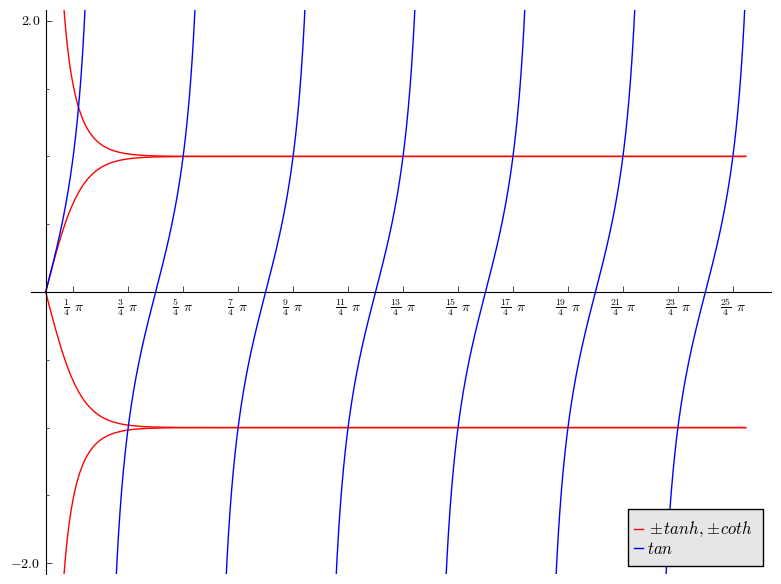}
  \caption{The Steklov eigenvalues of a square. Each intersection
    corresponds to a double eigenvalue.}
\label{figure:square}
\end{figure}
\begin{figure}[h]
  \centering
  \includegraphics[width=7cm]{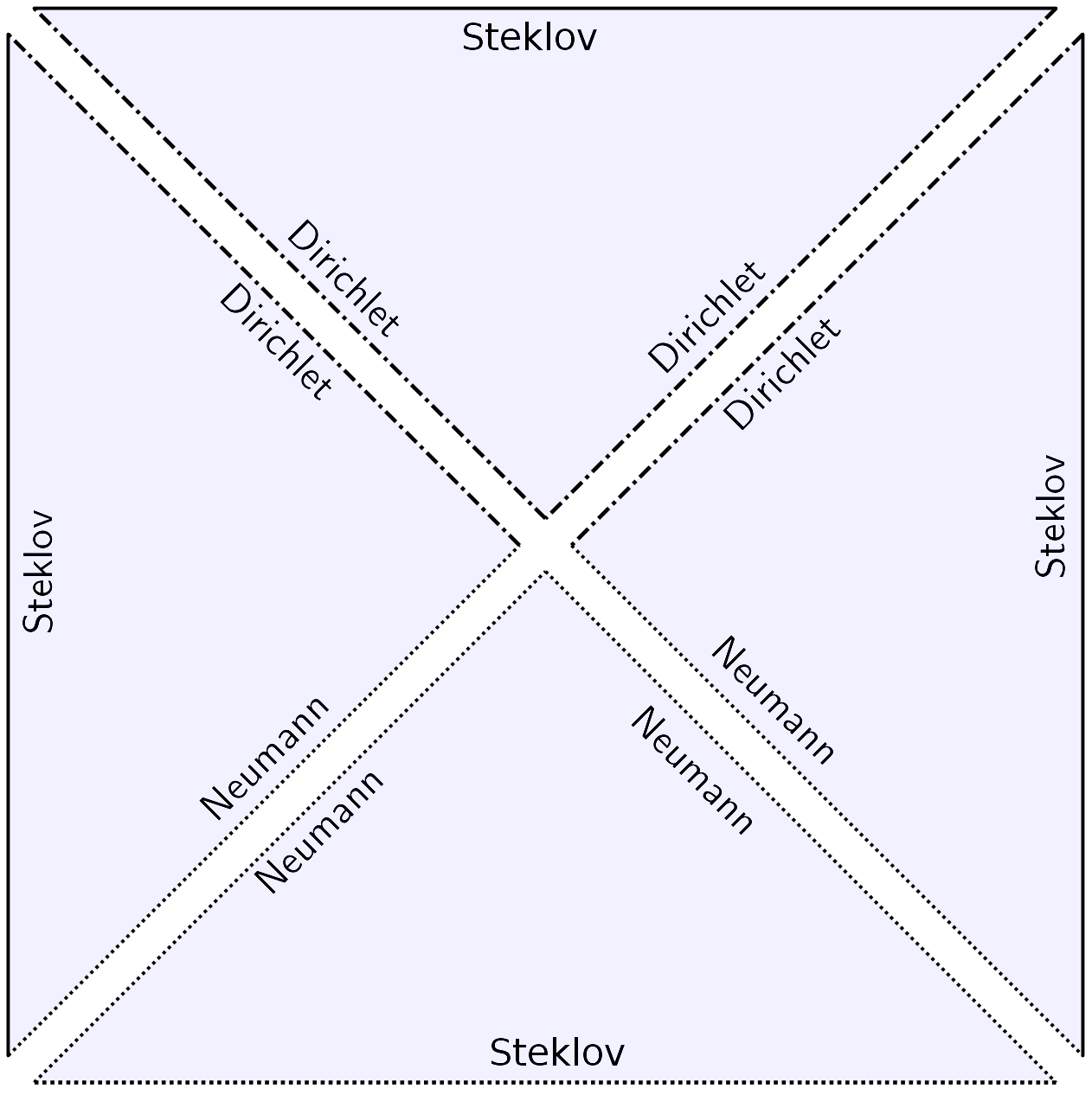}
  \caption{Decomposition of the Steklov problem on a square into four mixed
    problems on a triangle.}
\label{figure:squaresymmetries}
\end{figure}
One now needs  to prove the  completeness of this
system of orthogonal functions in $L^2(\partial\Omega)$. Using the
diagonal symmetries of the square (see Figure~\ref{figure:squaresymmetries}), we obtain symmetrized functions spanning the
same eigenspaces.  Splitting the eigenfunctions into odd and even with
respect to the diagonal symmetries, we represent the spectrum as the
union of the spectra of four mixed Steklov problems on a right
isosceles triangle. In each of these problems the Steklov condition is
imposed on the hypotenuse, and on each of the sides the condition is
either Dirichlet or Neumann, depending on whether the corresponding
eigenfunctions are odd or even when reflected across this side.  In
order to prove the  completeness of this system of Steklov
eigenfunctions, it is sufficient to show that the corresponding
symmetrized eigenfunctions form a complete set of solutions for each
of the four mixed problems.

Let us show this property for the problem corresponding to even symmetries
across the diagonal.  In this way, one gets a sloshing (mixed
Steklov--Neumann) problem on a right isosceles triangle. Solutions of
this problem were known since 1840s (see \cite{Lamb}). The
restrictions of the solutions to the hypotenuse (i.e. to the side of
the original square) turn out to be  the eigenfunctions of the
\emph{free beam equation}: 
\begin{align*}
  \frac{d^4}{dx^4}f=\omega^4f\qquad&\mbox{ on }(-1,1)\\
  \frac{d^3}{dx^3}f=\frac{d^2}{dx^2}f=0\quad&\mbox{ at }x=-1,1.
\end{align*}
This is a fourth order self-adjoint Sturm-Liouvillle equation. It is
known that its solutions form a complete set of functions on the interval
$(-1,1)$.  

The remaining three mixed problems are dealt with  similarly: one
reduces the problem to the study of solutions of the vibrating beam
equation with either the Dirichlet condition on both ends, or
the Dirichlet condition on one end and the Neumann on the other.

\subsection{Numerical experiments} Understanding fine spectral asy\-mptotics for the Steklov problem on
arbitrary polygonal domains is a difficult question. We have used
software from the FEniCS Project (see http://fenicsproject.org/ and \cite{fenicsbook})
to investigate the behaviour of the  
Steklov eigenvalues for some specific examples. This was done using an
implementation due to B. Siudeja~\cite{bartekweb} which was already
applied in~\cite{KKKNPPS}. For the sake of completeness, we discuss two
of these experiments here.

\begin{example}(Equilateral triangle)
  We have computed the first 60 normalized eigenvalues $\sigma_jL$ of
  an equilateral triangle. The results lead to a conjecture that
  $$\sigma_{2j}L=\sigma_{2j+1}L+\oo(1)=2\pi j+\oo(1).$$
\end{example}

\begin{example}(Right isosceles triangle)
  For the right isosceles triangle with  sides of lengths
  $1,1,\sqrt{2}$, we have also computed the first 60 normalized
  eigenvalues. The numerics indicate that the spectrum is
  composed of two sequences of eigenvalues, one of is which behaving as a
  sequence of double eigenvalues
  $$\pi j+\oo(1)$$ and the other one as a sequence of simple eigenvalues
  $$\frac{\pi}{\sqrt{2}}(j+1/2)+\oo(1).$$
\end{example}


In the context of the sloshing problem, some related  conjectures  have
been proposed in~\cite{foxkut}.



\section{Geometric inequalities for Steklov eigenvalues}\label{section:geometricinequalities}
\subsection{Preliminaries}
Let us start with the following simple observation. if a Euclidean domain $\Omega\subset\mathbb{R}^n$ is scaled by a factor $c>0$, then
\begin{gather}\label{sigmakscaling}
  \sigma_k(c\,\Omega)=c^{-1}\sigma_k(\Omega).
\end{gather}
Because of this scaling property, maximizing  $\sigma_k(\Omega)$ among domains with
fixed perimeter is equivalent to  maximizing  the normalized eigenvalues
$\sigma_k(\Omega)|\partial\Omega|^{1/{(n-1)}}$ on arbitrary domains. Here and further on we use the notation $|\cdot|$ to denote the volume of a manifold.

All the results concerning geometric bounds are proved using a  variational
characterization of the eigenvalues.  Let
$\mathcal{E}(k)$ be the set of all $k$ dimensional subspaces of the
Sobolev space $H^1(\Omega)$ which are orthogonal to constants on the
boundary $\partial\Omega$, then 
\begin{gather}\label{varcharsigmak}
  \sigma_k(\Omega,g)=\min_{E\in\mathcal{E}(k)}\sup_{0\neq u\in E}R(u),
\end{gather}
where the \emph{Rayleigh quotient} is
$$R(u)=\frac{\int_{\Omega}|\nabla u|^2\,dA}{\int_{M} u^2\,dS}.$$
In particular, the first nonzero eigenvalue is given by
\begin{gather*}
  \sigma_1(\Omega)=\min\Bigl\{R(u)\,:\,u\in H^1(\Omega),\,\int_{\partial\Omega}u\,dS=0\Bigr\}.
\end{gather*}
These variational characterizations are similar
to those of Neumann eigenvalues on $\Omega$, where the integral in the
denominator of $R(u)$ would be on the domain $\Omega$ rather than on
its boundary. 

One last observation is in order before we discuss isoperimetric
bounds. Let $\Omega_\epsilon:=(-1,1)\times(-\epsilon,\epsilon)$ be a
thin rectangle ($0<\epsilon<<1$). 
It is easy to see using using~(\ref{varcharsigmak}) that
\begin{gather}
\label{collapse}
  \lim_{\epsilon\rightarrow 0}\sigma_k(\Omega_\epsilon)=0,\qquad\mbox{ for each }k\in\mathbb{N}.
\end{gather}
In fact, it suffices for a family $\Omega_\epsilon$ of domains to have
a \emph{thin collapsing passage} (see Figure~\ref{figure:passage}) to
guarantee that $\sigma_k(\Omega_\epsilon)$ becomes arbitrarily small as
$\epsilon\searrow 0$ (see \cite[section 2.2]{gp}.)
\begin{figure}
  \centering
  \includegraphics[width=6cm]{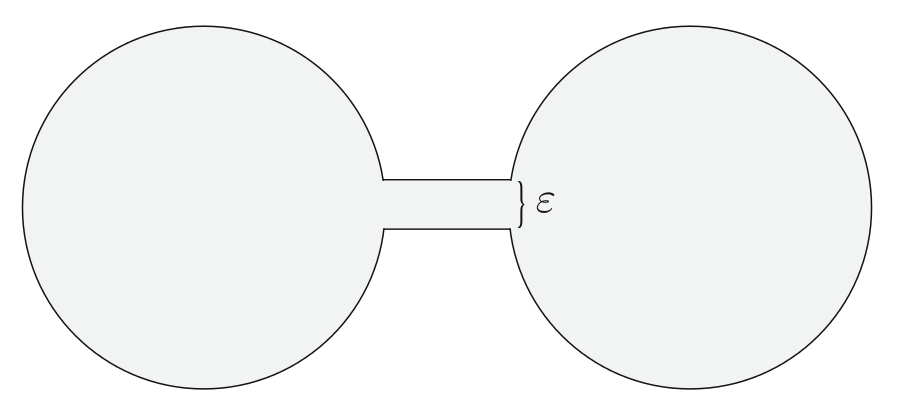}
  \caption{A domain  with a thin passage.}
  \label{figure:passage}
\end{figure}
This follows from the variational characterization: the idea is to construct a sequence of $k$ pairwise orthogonal test functions that oscillate 
inside the thin passage and vanish outside.  Then the Dirichlet energy
of such functions will be very small, while the denominator in the
Rayleigh quotient remains bounded away from zero, due to the
integration over the side of the passage. Hence, the Rayleigh quotient
will tend to zero, yielding \eqref{collapse}.
When considering an isoperimetric constraint, it is therefore more
interesting to maximize eigenvalues.

\subsection{Isoperimetric upper bounds for Steklov eigenvalues on
  surfaces}
\label{isoperim}
On a compact surface with boundary, the following theorem gives a
general upper bound in terms of the genus and the number of boundary components.
\begin{theorem}[\cite{gp3}]\label{thmHPSgenus}
  Let $\Omega$ be a smooth orientable compact surface with boundary
  $M=\partial\Omega$ of length $L$. Let $\gamma$ be the genus of $\Omega$ and let
  $l$ be the number of its boundary components. Then the following holds:
  \begin{gather}\label{InequalityGPHPS}
    \sigma_p\sigma_q\,L^2\leq
    \begin{cases}
      \pi^2(\gamma+l)^2 (p+q)^2 &\mbox{ if }p+q\mbox{ is even},\\
      \pi^2(\gamma+l)^2  (p+q-1)^2 &\mbox{ if }p+q\mbox{ is odd},
    \end{cases}
  \end{gather}
  for any pair of integers $p,q \ge 1$.
  In particular by setting $p=q=k$ one obtains the following bound:
  \begin{gather}\label{InequalityHPSgenusSimple}
    \sigma_k(\Omega)L(M)\leq 2\pi(\gamma+l)k.
  \end{gather}
\end{theorem}
The proof of Theorem~\ref{thmHPSgenus} is based on the existence of a
proper holomorphic covering map $\phi:\Omega\rightarrow\mathbb{D}$ of
degree $\gamma+l$ (the Ahlfors map), which was proved in~\cite{gabard}, and on an ingenious complex analytic argument due to
J. Hersch, L. Payne and M. Schiffer~\cite{hps}, who used it to prove
inequality~(\ref{InequalityGPHPS}) for planar domains. In this
particular case, inequality~(\ref{InequalityHPSgenusSimple}) is
known to be sharp. Indeed, it was proved in~\cite{gp} that equality is
attained in the limit by a family $\Omega_\epsilon$ of domains
degenerating to a disjoint union of $k$ identical disks (see Figure~\ref{figure:pullapart}).
\begin{figure} 
  \centering
  \includegraphics[width=6cm]{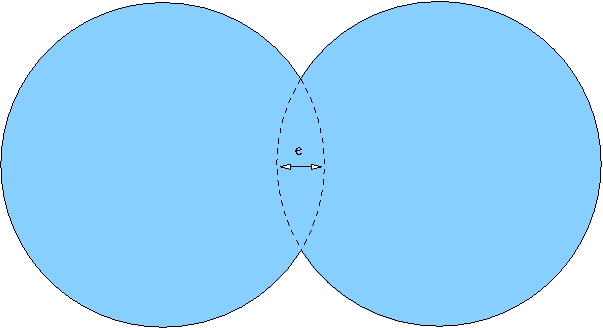}
  \caption{A family of domains $\Omega_\epsilon$ maximizing $\sigma_2L$
    in the limit as $\epsilon \to 0$.}
  \label{figure:pullapart}
\end{figure}
For $k=1$, inequality (\ref{InequalityHPSgenusSimple}) was proved in\cite{fraschoen2}.

The earliest isoperimetric inequality for Steklov eigenvalues is that of
R. Weinstock~\cite{wein}. For simply--connected planar domains ($\gamma=0, l=1$), he proved that 
\begin{gather}\label{Inequality:Weinstock}
  \sigma_1(\Omega)L(\partial\Omega)\leq2\pi
\end{gather}
with equality if and only if $\Omega$ is a disk.
Weinstock used an
argument similar to that of G. Szeg\H{o}~\cite{szego}, who obtained an isoperimetric inequality for the first nonzero Neumann
eigenvalue of a simply--connected domain $\Omega$ normalized by the
measure $|\Omega|$ rather than its perimeter.  In fact, Weinstock's
proof is the simplest application of the center of mass
renormalization (also known as Hersch's lemma, see ~\cite{Hersch,
  SchoenYau, GNP, gp2}).

While Szeg\H{o}'s inequality can be generalized to an arbitrary Euclidean domain
(see \cite{weinberger}), this is not true for Weinstock's
inequality. In particular, as follows from the example below,
Weinstock's inequality fails for non-simply--connected planar domains. 
\begin{example}
The Steklov eigenvalues and eigenfunctions of an annulus have been computed in
\cite{dittmar1}.
On the annulus $A_\epsilon=\mathbb{D}\setminus B(0,\epsilon)$, there
is a radially symmetric Steklov
eigenfunction
$$f(r)=-\left(\frac{1+\epsilon}{\epsilon\log(\epsilon)}\right)\log(r)+1,$$
with the corresponding eigenvalue
$\sigma=\frac{1+\epsilon}{\epsilon\log(1/\epsilon)}$.
All other
eigenfunctions are of the form
$$f_k(r,\theta)=(A_{k}r^k+A_{-k}r^{-k})T(k\theta)\qquad (\mbox{with }k\in\mathbb{N})$$
where $T(k\theta)=\cos(k\theta)$ or $T(k\theta)=\sin(k\theta)$. In order for
$f_k(r,\theta)$ to be a Steklov eigenfunction it is required that
$$\frac{\partial}{\partial_r}f_k(1,\theta)=\sigma
f(1,\theta)\qquad\mbox{ and}\qquad -\frac{\partial}{\partial_r}f_k(\epsilon,\theta)=\sigma
f(\epsilon,\theta),$$ which leads to the following system: 
$$
\left(
  \begin{array}{cc}
  \sigma\epsilon^k+k\epsilon^{k-1}&\sigma\epsilon^{-k}-k\epsilon^{-k-1}\\
  \sigma-k&\sigma+k
\end{array}
\right)
\left(
  \begin{array}{c}
    A_k\\
    A_{-k}
  \end{array}
\right)
=
\left(
  \begin{array}{c}
    0\\
    0
  \end{array}
\right).
$$
This system has a non-zero solution if and only if its determinant vanishes. 
After some simplifications, the Steklov eigenvalues of the annulus $A_\epsilon=\mathbb{D}\setminus B(0,\epsilon)$
are seen to be the roots of the quadratic polynomials
$$p_k(\sigma)=\sigma^2
  -\sigma k\left(\frac{\epsilon+1}{\epsilon}\right)\left(\frac{1+\epsilon^{2k}}{1-\epsilon^{2k}}\right)
  +\frac{1}{\epsilon}k^2\qquad (k\in\mathbb{N}).$$
Each of these eigenvalues is double, corresponding to the choice of a
$\cos$ or $\sin$ function for the angular part $T(k\theta)$ of the corresponding eigenfunction.
For $\epsilon>0$ small enough, this leads in particular to
\begin{gather}\label{sigmaoneannulus}
  \sigma_1(A_\epsilon)=\frac{1}{2\epsilon}\frac{1+\epsilon^{2}}{1-\epsilon}\left(1-\sqrt{1-4\epsilon\left(\frac{1-\epsilon}{1+\epsilon^{2}}\right)^2}\right).
\end{gather}
\end{example}
It follows from formula~(\ref{sigmaoneannulus}) that for the
annulus $A_\epsilon=B(0,1)\setminus B(0,\epsilon)$  one has 
\begin{gather}\label{AsymptoticHole}
  \sigma_1(A_\epsilon)L(\partial
  A_\epsilon)=2\pi\sigma_1(\mathbb{D})+2\pi\epsilon+o(\epsilon)\quad\mbox{as
    }\epsilon\searrow 0.
\end{gather}
Therefore,  $\sigma_1(A_\epsilon)L(\partial
A_\epsilon)>2\pi\sigma_1(\mathbb{D})$ for $\epsilon>0$ small
enough (see Figure~\ref{figure:hole}), and hence Weinstock's inequality \eqref{Inequality:Weinstock} fails.
\begin{figure}
  \centering
  \includegraphics[width=6cm]{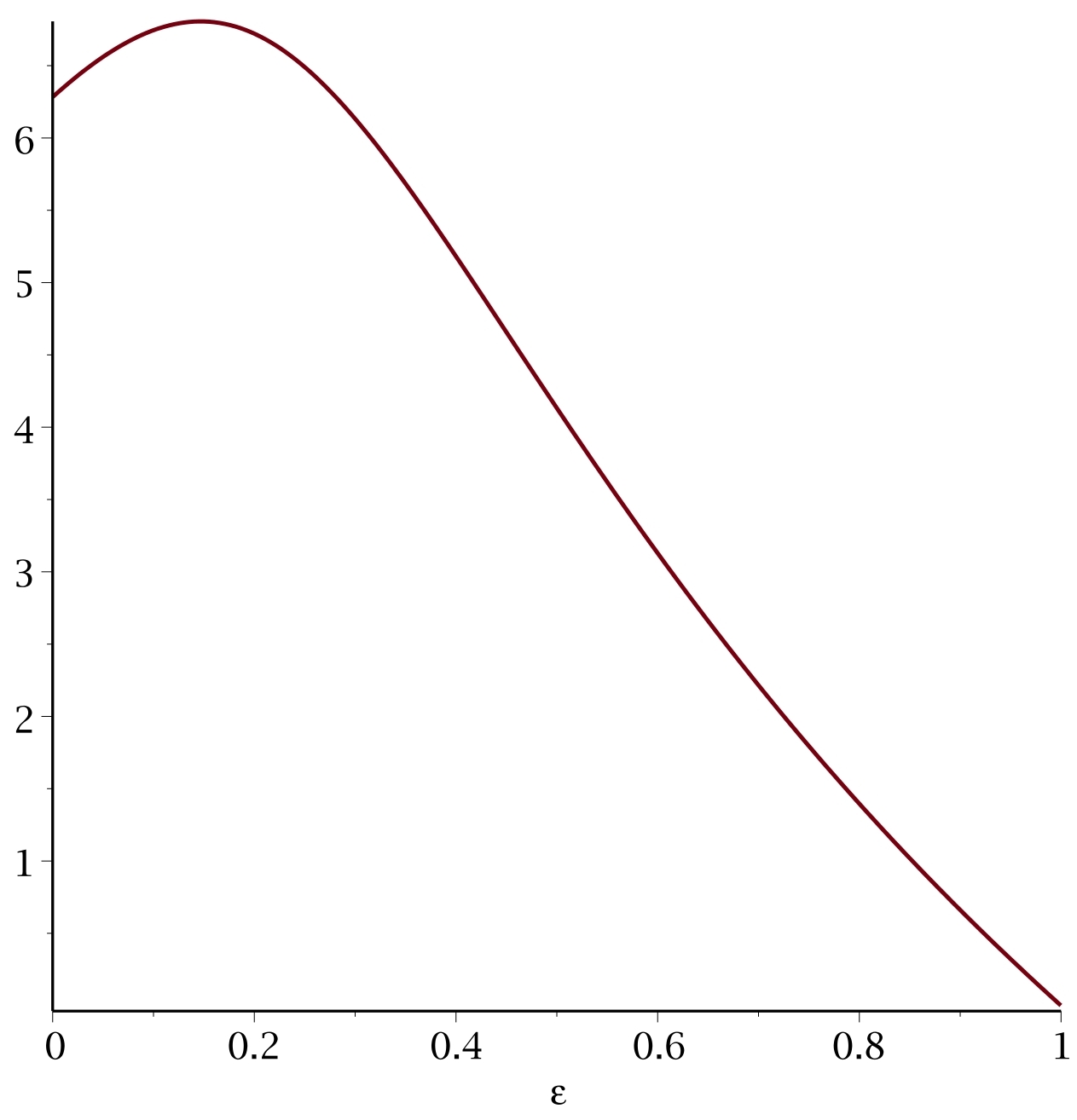}
  \caption{The normalized eigenvalue $\sigma_1(A_\epsilon)L(\partial A_\epsilon)$}
  \label{figure:hole}
\end{figure}
\begin{remark} One can also compute
  the Steklov eigenvalues of the spherical shell
  $\Omega_\epsilon:=B(0,1)\setminus 
  B(0,\epsilon)\subset\mathbb{R}^{n}$ for $n\geq 3$. The eigenvalues
  are the roots of certain quadratic polynomials which can be computed explicitly.
Here again, it is true that for $\epsilon>0$ small enough,
$\sigma_1(\Omega_\epsilon)|\partial\Omega_\epsilon|^{\frac{1}{n-1}}>\sigma_1(\mathbb{B})|\partial\mathbb{B}|^{\frac{1}{n-1}}$. This
computation was part of an undergraduate research project of E. Martel at
Universit\'e Laval.
\end{remark}

Given that Weinstock's inequality is no longer true for non-simply--connected planar domains, one may ask whether the supremum
of $\sigma_1L$ among all planar domains of fixed perimeter is
finite. This is indeed the case, as follows from the following theorem for $k=1$ and $\gamma=0$.
\begin{theorem}[\cite{ceg2}]
\label{unibound} 
 There exists a universal constant $C>0$ such that
  \begin{gather}\label{InequalityNoBoundary}
    \sigma_k(\Omega)L(\partial\Omega)\leq C(\gamma+1)k.
  \end{gather}
\end{theorem}
Theorem \ref{unibound} leads to the the following question:
\begin{open}
\label{ddd}
  What is the maximal value of $\sigma_1(\Omega)$ among Euclidean
  domains $\Omega\subset\mathbb{R}^n$ of fixed perimeter? On which domain (or in the limit of which sequence of domains) is it realized?  
\end{open}
Some related results will be discussed in subsection
\ref{Section:FS}. In particular, in view of  Theorem
\ref{thm:fraschoengenus0}~\cite{fraschoen2}, it is tempting to suggest that
the maximum is realized in the limit by a sequence of domains with the
number of boundary components tending to infinity.

The proof of Theorem \ref{unibound} is based on N. Korevaar's metric geometry
approach~\cite{kvr} as described in~\cite{gyn}. For $k=1$,
inequality~(\ref{InequalityNoBoundary}) holds with 
$C=8\pi$ (see~\cite{kok}). For $k=1$ and $\gamma=0$, it holds with
$C=4\pi$~\cite{fraschoen2} (see Theorem~\ref{thm:fraschoengenus0} below). 
It is also possible to ``decouple'' the genus $\gamma$ and the index
$k$. The following theorem was proved by
A. Hassannezhad~\cite{hassannezhad}, using a generalization of the
Korevaar method in combination with concentration results
from~\cite{ColbMaert}.
\begin{theorem}
  There exists two constants $A,B>0$ such that
  $$\sigma_k(\Omega)L(\partial\Omega)\leq A\gamma+Bk.$$
\end{theorem}

At this point, we have considered maximization of the Steklov
eigenvalues under the constraint of fixed perimeter. 
This is natural, since they are the eigenvalues of
to the Dirichlet-to-Neumann operator, which acts on the
boundary. Nevertheless, it is also possible to normalize the 
eigenvalues by fixing the measure of $\Omega$.
 The following theorem
was proved by F. Brock~\cite{brock}. 
\begin{theorem}\label{thm:brock}
  Let $\Omega\subset\mathbb{R}^{n}$ be a bounded Lipschitz
  domain. Then
  \begin{gather}\label{Inequality:Brock}
    \sigma_1(\Omega)|\Omega|^{1/n}\leq \omega_n^{1/n},
  \end{gather}
  with equality if and only if $\Omega$ is a ball. Here $\omega_n$ is
  the volume of the unit ball $\mathbb{B}^n\subset\mathbb{R}^n$.
\end{theorem}
Observe that no connectedness assumption is required this time.
The  proof of Theorem~\ref{thm:brock} is based on a weighted
isoperimetric inequality for moments of inertia of the boundary
$\partial\Omega$.
A quantitative improvement of Brock's theorem was obtained in~\cite{brasdephilruffini}
in terms of the  \emph{Fraenkel asymmetry} of a
bounded domain $\Omega\subset\mathbb{R}^n$:
$$\mathcal{A}(\Omega):=\inf
\left\{
  \frac{\|1_\Omega-1_B\|_{L^1}}{|\Omega|}\,:\, B\mbox{ is a ball with
  } |B|=|\Omega|
\right\}.$$
\begin{theorem}\label{thm:stabilitySteklov}
  Let $\Omega\subset\mathbb{R}^{n}$ be a bounded Lipschitz
  domain. Then
  \begin{gather}\label{Stability}
    \sigma_1(\Omega)|\Omega|^{1/n}\leq \omega_n^{1/n}(1-\alpha_n\mathcal{A}(\Omega)^2),
  \end{gather}
  where $\alpha_n>0$ depends only on the dimension.
\end{theorem}
The proof of Theorem~\ref{thm:stabilitySteklov} is based on a
quantitative refinement of the isoperimetric inequality. It would be interesting to prove a
similar stability result for Weinstock's inequality:
\begin{open}
  Let $\Omega$ be a planar simply--connected domain such that
the difference  $2\pi-\sigma_1(\Omega)L(\partial\Omega)$ is small. Show that $\Omega$ must be close
  to a disk (in the sense of Fraenkel asymmetry or some other measure of proximity).
\end{open}

\subsection{Existence of maximizers and free boundary minimal
  surfaces}\label{Section:FS}
 A \emph{free boundary submanifold} is a proper minimal submanifold of some unit ball
$\mathbb{B}^n$ with its boun\-dary meeting the sphere $\mathbb{S}^{n-1}$
orthogonally. These are characterized by their Steklov eigenfunctions.
\begin{lemma}[\cite{fraschoen}]
\label{lemmafs}
  A properly immersed submanifold  $\Omega$ of the ball $\mathbb{B}^n$
  is a free boun\-dary submanifold if and only if the restriction to
  $\Omega$ of the coordinate functions $x_1,\cdots,x_n$ satisfy
  \begin{equation*}
    \begin{cases}
      \Delta x_i=0& \mbox{ in } \Omega,\\

      \partial_\nu x_i= \,x_i& \mbox{ on }\partial\Omega.
    \end{cases}
  \end{equation*}
\end{lemma}

This link was exploited by A. Fraser and
R. Schoen who developed the theory of extremal metrics for Steklov
eigenvalues. See~\cite{fraschoen,fraschoen2} and
especially~\cite{fraschoen3} where an overview is presented.

Let $\sigma^\star(\gamma,k)$ be the supremum of $\sigma_1L$ taken over
all Riemannian metrics on a compact surface of genus $\gamma$ with $l$
boundary components. In~\cite{fraschoen2}, a geometric
characterization of maximizers was proved.
\begin{proposition}\label{prop:fraschoencritmetric}
  Let $\Omega$ be a compact surface of genus $\gamma$ with $l$
  boundary components and let $g_0$ be a 
  smooth metric on $\Omega$ such that
  $$\sigma_1(\Omega,g_0)L(\partial\Omega,g_0)=\sigma^\star(\gamma,l).$$
  Then there exist eigenfunctions $u_1,\cdots,u_n$ corresponding to
  $\sigma_1(\Omega)$ such that
  the map 
  $$u=(u_1,\cdots,u_n):\Omega\rightarrow\mathbb{B}^n$$
  is a conformal minimal immersion such that $u(\Omega)\subset\mathbb{B}^n$ is
  a free boundary solution, and 
  is an isometry on $\partial\Omega$ up to a rescaling by a constant factor.
\end{proposition}
This result was extended to higher eigenvalues $\sigma_k$
in~\cite{fraschoen3}. This characterization is similar to that of
extremizers of the eigenvalues of the Laplace operator on surfaces
(see~\cite{nad,MR1781616,MR2378458}).

For  surfaces of genus zero, Fraser and Schoen could also obtain an
existence and regularity result
for maximizers, which is the  main result of their paper~\cite{fraschoen2}.
\begin{theorem}\label{thm:fraschoenexistence}
  For  each $l>0$, there exists a smooth metric $g$ on the surface of
  genus zero  with $l$  boundary components such that
  $$\sigma_1(\Omega,g)L_g(\partial\Omega)=\sigma^\star(0,l).$$
\end{theorem}
Similar existence results have been proved for the first nonzero
eigenvalue of the Laplace--Beltrami operator in a fixed conformal
class of a closed surface of
arbitrary genus, in which case conical singularities have to be allowed
(see~\cite{JLNNP, petrides}). 

Proposition~\ref{prop:fraschoencritmetric} and
Theorem~\ref{thm:fraschoenexistence} can be used to study optimal
upper bounds for $\sigma_1$ on surfaces of genus zero. Observe that
inequality~(\ref{InequalityHPSgenusSimple}) can be restated as 
$$\sigma^\star(\gamma,l)\leq 2\pi(\gamma+l).$$
This bound is not sharp in general. For
instance, Fraser and Schoen~\cite{fraschoen2} proved that on annuli
($\gamma=0, l=2$), the maximal value of
$\sigma_1(\Omega)L(\partial\Omega)$ is attained by the
\emph{critical catenoid} ($\sigma_1L\sim 4\pi/1.2$), which is
the minimal surface $\Omega\subset B^3$ parametrized by
$$\phi(t,\theta)=c(\cosh(t)\cos(\theta),\cosh(t)\sin(\theta),t),$$
where the scaling factor $c>0$ is chosen so that the boundary of the
surface $\Omega$ meets the sphere $\mathbb{S}^2$ orthogonally.
\begin{theorem}[\cite{fraschoen2}]\label{thm:fraschoencritcat}
  The supremum of $\sigma_1(\Omega) L(\partial\Omega)$ among surfaces
  of genus 0 with two boundary components is attained by the
  \emph{critical catenoid}. The maximizer is unique
  up to conformal changes of the metric  which are constant on the
  boundary.
\end{theorem}
The uniqueness statement is proved using
Proposition~\ref{prop:fraschoencritmetric} by showing that the critical
catenoid is the unique free boundary annulus in a Euclidean ball.
The maximization of $\sigma_1L$ for the M\"obius bands has also
been considered in~\cite{fraschoen2}.

For surfaces of genus zero with arbitrary number of boundary components,
the maximizers are not known explicitly, but the asymptotic behaviour for large
number of boundary components is understood~\cite{fraschoen2}.
\begin{theorem}\label{thm:fraschoengenus0}
  The sequence $\sigma^\star(0,l)$ is strictly increasing and
  converges to $4\pi$. For  each $l\in\mathbb{N}$ a maximizing metric
  is achieved by a free boundary minimal surface $\Omega_l$ of area
  less than $2\pi$. The limit of these minimal surfaces as
  $l\nearrow+\infty$ is a double disk.
\end{theorem}

The results discussed above lead to the following question:

\begin{open}
  Let $\Omega$ be a surface of genus $\gamma$ with $l$ boundary
  components. Does there exist a smooth Riemannian metric $g_0$ such that
  $$\sigma_1(\Omega,g_0)L(\partial\Omega,g_0)\geq\sigma_1(\Omega,g)L(\partial\Omega,g)$$
  for each Riemannian metric $g$?
\end{open}

Free boundary minimal surfaces were used as a tool in the
study of maximizers for $\sigma_1$, but this interplay can be turned around and
used to obtain interesting geometric results.
\begin{corollary}
  For each $l\geq 1$,  there exists an embedded minimal surface of genus
  zero in $\mathbb{B}^3$ with $l$ boundary components satisfying the
  free boundary condition.
\end{corollary}



\subsection{Geometric bounds in higher dimensions}

In dimensions $n=\mbox{dim}(\Omega)\geq 3$, isoperimetric inequalities for Steklov eigenvalues are
more complicated, as they involve other geometric quantities, such as the isoperimetric ratio:
$$I(\Omega)=\frac{|M|}{|\Omega|^{\frac{n-1}{n}}}.$$
For the first nonzero eigenvalue $\sigma_1$, it is possible to obtain
upper bounds for general compact manifolds with
boundary in terms of $I(\Omega)$ and of the relative conformal volume, which is defined below.
Let $\Omega$ be a compact manifold of dimension $n$ with smooth 
boundary $M$. Let $m\in\mathbb{N}$ be a positive integer. The
\emph{relative $m$-conformal volume} of $\Omega$ is
\begin{gather*}
  V_{rc}(\Omega,m)=\inf_{\phi:\Omega\hookrightarrow B^m}\sup_{\gamma\in M(m)}\mbox{Vol}(\gamma\circ\phi(\Omega)),
\end{gather*}
where the infimum is over all conformal immersions
$\phi:\Omega\hookrightarrow\mathbb{B}^m$ such that
$\phi(M)\subset\partial\mathbb{B}^m$, and $M(m)$ is the group of
conformal diffeomorphisms of the ball. 
This conformal invariant was introduced in ~\cite{fraschoen}. It is similar to the celebrated conformal volume of
P. Li and S.-T. Yau~\cite{liyau}. 
\begin{theorem}\cite{fraschoen}
  Let $\Omega$ be a compact Riemannian manifold of dimension $n$
  with smooth boundary $M$. For each positive integer $m$, the
  following holds:
  \begin{gather}\label{IneqFSconf}
    \sigma_1(\Omega)|M|^{\frac{1}{n-1}}\leq \frac{nV_{rc}(\Omega,m)^{2/n}}{I(\Omega)^{\frac{n-2}{n-1}}}.
  \end{gather}
  In case of equality, there exists a conformal harmonic map
  $\phi:\Omega\rightarrow\mathbb{B}^m$ which is a homothety on $M=\partial\Omega$ and such that $\phi(\Omega)$ meets $\partial B^m$
  orthogonally. If $n\geq 3$, then $\phi$ is is an isometric minimal
  immersion of $\Omega$ and it is given by a subspace of the first
  eigenspace.
\end{theorem}
The proof uses coordinate functions as test functions and is based on
the Hersch center of mass renormalization procedure. It is similar to
the proof of the Li-Yau inequality~\cite{liyau}.

For higher eigenvalues, the following upper bound for bounded domains was proved by 
B. Colbois, A. El Soufi and the first author in~\cite{ceg2}.
\begin{theorem}\label{CEGconf}
  Let $N$ be a Riemannian manifold of dimension $n$. If $N$ is
  conformally equivalent to a complete Riemannian manifold with
  non-negative Ricci curvature, then for each domain $\Omega\subset
  N$, the following holds for each $k\geq 1$,
  \begin{gather}\label{IneqCEGconf}
    \sigma_k(\Omega)|M|^{\frac{1}{n-1}}\leq\frac{\alpha(n)}{I(\Omega)^{\frac{n-2}{n-1}}}k^{2/n}.
  \end{gather}
  where $\alpha(n)$ is a constant depending only $n$.
\end{theorem}
The proof of Theorem~\ref{CEGconf} is based on the methods of  metric geometry
initiated in~\cite{kvr} and further developed in~\cite{gyn}.
In combination with the classical
isoperimetric inequality, Theorem~\ref{CEGconf} leads to the following corollary.
\begin{cor}\label{Cor:boundedeuclidean}
  There exists a constant $C_n$ such that for any Euclidean domain
  $\Omega\subset\mathbb{R}^{n}$ 
  $$\sigma_k(\Omega)|\partial\Omega|^{\frac{1}{n-1}}\leq C_n k^{2/n}.$$
\end{cor}
Similar results also hold for domains in the hyperbolic space $\mathbb{H}^{n}$ and
in the hemisphere of $\mathbb{S}^{n}$.
An interesting  question raised in~\cite{ceg2}  is whether one can replace the exponent  $2/n$ in Corollary \ref{Cor:boundedeuclidean} by 
 $1/(n-1)$, which should be optimal in view of Weyl's law \eqref{Weylaw}:
\begin{open}
  Does  there exist a constant $C_n$ such that any bounded Euclidean
  domain $\Omega\subset\mathbb{R}^n$ satisfies
  $$\sigma_k(\Omega)|\partial\Omega|^{\frac{1}{n-1}}\leq C_nk^{\frac{1}{n-1}}?$$
\end{open}
While it might be tempting to
think that inequality~(\ref{IneqCEGconf}) should also hold with the exponent
$1/(n-1)$,  this is false since it  would imply a universal upper
bound on the isoperimetric ratio $I(\Omega)$ for Euclidean domains.

\subsection{Lower bounds}

In~\cite{esco1},  J. Escobar proved the following lower bound.
\begin{theorem}
  Let $\Omega$ be a smooth compact Riemannian manifold of
  dimension $\geq 3$ with boundary $M=\partial\Omega$. Suppose that
  the Ricci curvature of $\Omega$
  is non-negative and that the second fundamental form of $M$ is
  bounded below by $k_0>0$, then
  $\sigma_1>k_0/2.$
\end{theorem}
The proof is a simple application of Reilly's formula. 
In~\cite{esco2}, Escobar conjectured the stronger bound
$\sigma_1\geq k_0$, which he proved for surfaces. For convex planar domains,
this had already been proved by Payne~\cite{payne}. Earlier lower
bounds for convex and starshaped planar domains are due to Kuttler and
Sigillito~\cite{kuttsigi2,kuttsigi3}. 

In more general situations (e.g. no convexity assumption), it is still
possible to bound the first eigenvalue from below, similarly to
the classical Cheeger inequality. The classical Cheeger constant associated
to a compact Riemannian manifold $\Omega$ with boundary
$M=\partial\Omega$ is defined by 
\begin{gather*}
  h_{c}(\Omega):=\inf_{|A|\leq\frac{|\Omega|}{2}}\frac{|\partial A\cap
    \mbox{int }\Omega|}{|A|}.
\end{gather*}
where the infimum is over all Borel subsets of $\Omega$ such that
$|A|\leq |\Omega|/2$.
In~\cite{jammes1} P.~Jammes introduced the following Cheeger type
constant for the Steklov problem:
\begin{gather*}
  h_{\mbox{j}}(\Omega):=\inf_{|A|\leq\frac{|\Omega|}{2}}\frac{|\partial A\cap
    \mbox{int }\Omega|}{|A\cap\partial\Omega|}.
\end{gather*}
He proved the following lower bound.
\begin{theorem}\label{ThmCheegerJammes}
  Let $\Omega$ be a smooth compact Riemannian manifold with boundary
  $M=\partial\Omega$. Then
  \begin{gather}\label{InequalityCheegerJammes}
    \sigma_1(\Omega)\geq \frac{1}{4}h_c(\Omega)h_j(\Omega)
  \end{gather}
\end{theorem}
The proof of this theorem uses the coarea formula and follows the proof of the classical Cheeger
inequality quite closely.
Previous lower bounds were also obtained in~\cite{esco1} in terms of
a related Cheeger type constant and of the first eigenvalue of a Robin
problem on $\Omega$.

\subsection{Surfaces with large Steklov eigenvalues}
The previous discussion immediately raises the question of whether
there exist  surfaces with an arbitrarily large normalized first Steklov
eigenvalue.  The question was settled by the  first author and
B. Colbois in~\cite{colbgir1}.
\begin{theorem}\label{Theorem:CGLARGE}
  There exists a sequence $\{\Omega_N\}_{N\in\mathbb{N}}$ of compact surfaces
  with boundary and a constant $C>0$ such that for each $N\in\mathbb{N}$,
  $\mbox{genus}(\Omega_N)=1+N,$ and 
  $$\sigma_1(\Omega_N)L(\partial\Omega_N)\geq CN.$$
\end{theorem}
The proof is based on the construction of surfaces which are modelled on
a family of expander graphs.

\begin{remark}
The literature on geometric bounds for Steklov eigenvalues is
expanding rather fast. There is some interest in considering the maximization of various
functions of the Steklov  eigenvalues. See \cite{dittmar1,edward2,HenPhil}.
In the framework of comparison geometry, $\sigma_1$ was studied is~\cite{esco3} and more recently in~\cite{BinSan}. 
For submanifolds of $\mathbb{R}^n$, upper bounds involving the mean
curvatures of $M=\partial\Omega$ have been obtained in~\cite{IliasMakhoul}. Higher eigenvalues on annuli have
been studied in~\cite{FanTamYu}.
\end{remark}


\section{Isospectrality and spectral rigidity}\label{section:IsospectralityRigidity}
\label{isosp}
\subsection{Isospectrality and the Steklov problem}
Adapting the celebrated question of M. Kac ``Can one hear the shape of a drum?''  to the Steklov problem, one may ask:
\begin{open}
\label{isospdomains}
Do there exist planar domains which are not isometric and have the same Steklov spectrum?
\end{open}
We believe the answer to this question is negative. Moreover, the
problem can be viewed as a special case of a conjecture put forward in
\cite{JolSharaf}: two surfaces have the same Steklov spectrum if and
only if there exists a conformal mapping between them such that the
conformal factor on the boundary is identically equal to one.  Note
that the ``if'' part immediately follows from the variational principle
\eqref{varcharsigmak}. Indeed, the numerator of the Rayleigh quotient
for Steklov eigenvalues is the usual Dirichlet energy, which is
invariant under conformal transformations in two dimensions. The
denominator also stays the same if the conformal factor is equal to
one on the boundary. Therefore, the Steklov spectra of such
conformally equivalent surfaces coincide.

In higher dimensions, the Dirichlet energy is not conformally
invariant, and therefore the approach described above does not
work. However, one can construct Steklov isospectral manifolds of
dimension $ n\ge 3$ with the help of Example~\ref{cyl}. Indeed, given
two compact manifolds $M_1$ and $M_2$ which are Laplace-Beltrami
isospectral (there are many known examples of such pairs, see, for instance,  \cite{Buser1986, Sunada,GordonPerrySchueth}), consider two cylinders  $\Omega_1=M_1 \times
[0,L]$ and $\Omega_2 = M_2 \times [0,L]$, $L>0$. It follows from
Example~\ref{cyl} that $\Omega_1$ and $\Omega_2$ have the same Steklov
spectra. Recently, examples of higher-dimesional Steklov isospectral
manifolds with connected boundaries were announced
in~\cite{GordHerbWebb}.

In all known constructions of Steklov isospectral manifolds,
their boundaries are Laplace isospectral.  The following question was
asked in \cite{GPPS}:
\begin{open}
  Do there exist Steklov isospectral manifolds such that their
  boundaries are not Laplace isospectral?
 \end{open}

\subsection{Rigidity of the Steklov spectrum: the case of a ball} It
is an interesting and challenging question to find examples of
manifolds with boundary that are uniquely determined by their Steklov
spectrum. In this subsection we discuss the seemingly simple example
of Euclidean balls.
\begin{proposition}
  A disk is uniquely determined by its Steklov spectrum among all smooth
  Euclidean domains.
\end{proposition}
\begin{proof}
  Let $\Omega$ be an Euclidean domain which has the same Steklov
  spectrum as the disk of radius $r$. Then, by Corollary \ref{cor1}
  one immediately deduces that $\Omega$ is a planar domain  of
  perimeter $2\pi r$.  Moreover, it follows from Theorem \ref{xx} that
  $\Omega$ is simply--connected. Therefore, since the equality in
  Weinstock's inequality \eqref{Inequality:Weinstock} is achieved for
  $\Omega$, the domain $\Omega$ is a disk of radius $r$.
\end{proof}
\begin{remark} 
The smoothness hypothesis in the proposition above seems to be purely technical. We have to make this assumption since we make use of  Theorem \ref{xx}. 
\end{remark}
The above result motivates
\begin{open}
  Let $\Omega \subset {\mathbb R}^n$ be a domain which is isospectral to a ball of radius $r$. Show that it is a ball of radius $r$. 
\end{open}
Note that Theorem~\ref{thm:brock} does not yield a solution to this problem because the volume $|\Omega|$ is not a Steklov spectrum invariant.
Using the heat invariants of the Dirichlet-to-Neumann operator (see subsection \ref{specinv}), one can prove the following statement in dimension three. 
\begin{proposition}
  Let $\Omega\subset\mathbb R^3$ be a domain with connected and smooth
  boundary $M$. Suppose its Steklov spectrum is equal to that of
  a ball of radius $r$. Then $\Omega$ is a ball of radius $r$.
\end{proposition}
This result  was obtained  in~\cite{PoltSher}, and  we sketch its proof below. First,  let us  show that $M$ is simply--connected. We use an adaptation of a theorem of  Zelditch on
multiplicities~\cite{Zelditch} proved using microlocal analysis. 
 Namely,  since $\Omega$ is Steklov isospectral to a ball,  the multiplicities of its Steklov eigenvalues grow as $m_k=C k+\oo(1)$,
where  $C>0$ is some constant and $m_k$ is the multiplicity of the
$k$-th {\it distinct} eigenvalue (cf. Example~\ref{balls}).
Then one deduces that $M$ is a Zoll surface (that is, all geodesics on
$M$ are periodic with a common period), and hence it is simply--connected~\cite{Besse}. 

Therefore, the following
formula holds for the coefficient $a_2$ in the Steklov heat trace
asymptotics~\eqref{heatexp} on $\Omega$: 
\begin{gather*}
  a_2=\frac{1}{16\pi}\int_M H_1^2+\frac{1}{12}.
\end{gather*}
Here $H_1(x)$ denotes the mean curvature of $M$ at the point $x$, and  the term $\frac{1}{12}$ is obtained from the Gauss--Bonnet theorem using the fact
that $M$ is simply--connected. We have then:  $\int_M H_1^2=\int_{S_r} H_1^2$, where $S_r=\partial B_r$. 

On the other hand, it follows from \eqref{Weylaw} and Corollary
\ref{cor2} that  $\vol(M)$ and $\int_M H_1$ are Steklov spectral invariants. Therefore,
$$\area(M)=\area(S_r) ,\,\,\,\int_M H_1=\int_{S_r} H_1.$$
Hence
\begin{equation*}
\sqrt{\area(M)}\left(\int_M H_1^2\right)^{1/2}-\left|\int_M H_1\right|=\sqrt{\area(S_r)}\left(\int_{S_r} H_1^2\right)^{1/2}-\left|\int_{S_r} 
H_1\right|=0.\end{equation*}

Since the Cauchy-Schwarz inequality becomes an equality only for
constant functions, one gets that $H_1$ must be constant on $M$.  By a
theorem of Alexandrov \cite{Alexandrov}, the only compact surfaces of
constant mean curvature embedded in $\mathbb R^3$ are round
spheres. We conclude that $M$ is itself a sphere of radius $r$ and
therefore $\Omega$ is isometric to $B_r$. This completes the proof of the proposition.

\section{Nodal geometry and multiplicity bounds}\label{section:NodalMultiplicity}
\label{nodal}
\subsection{Nodal domain count}
The study of nodal domains and nodal sets of eigenfunctions is
probably the oldest topic in geometric spectral theory, going back to
the experiments of E. Chladni with vibrating plates. The fundamental
result in the subject is Courant's nodal domain theorem which states
that the $n$-th eigenfunction of the Dirichlet boundary value problem
has at most $n$ nodal domains. The proof of this statement uses
essentially two ingredients: the variational principle and the unique
continuation for solutions of second order elliptic equations. It can
therefore be extended  essentially verbatim  to Steklov eigenfunctions (see~\cite{kuttsigi2,KarKoPo}).
\begin{theorem}
  Let $\Omega$ be a compact Riemannian manifold with boundary and
  $\phi_n$ be an eigenfunction corresponding to the $n$-th nonzero
  Steklov eigenvalue  $\sigma_n$. Then $\phi_n$ has at most $n+1$
  nodal domains. 
\end{theorem}
Note that the Steklov spectrum starts with $\sigma_0=0$,  and
  therefore the $n$-th nonzero eigenvalue is actually the $n+1$-st Steklov
  eigenvalue.

Apart of the ``interior'' nodal domains and nodal sets of Steklov
eigenfunctions, a natural problem is to study the boundary nodal
domains and nodal sets, that is, the nodal domains and nodal sets of
the eigenfunctions of the Dirichlet-to-Neumann operator. 

The proof of Courant's theorem cannot be generalized to the Dirichlet-to-Neumann operator because it is nonlocal. The following problem therefore arises:
\begin{open}
\label{opennodal}
Let $\Omega$ be a Riemannian manifold with boundary $M$. Find an upper bound for the number of nodal domains of the $n$-th eigenfunction of the Dirichlet-to-Neumann operator on $M$.
\end{open}

For surfaces, a simple topological argument shows that the bound on
the number of interior nodal domains implies an estimate on the number
of boundary nodal domains of a Steklov eigenfunction.  In particular,
the $n$-th nontrivial Dirichlet-to-Neumann eigenfunction on the
boundary of a simply--connected planar domain has at most  $2n$  nodal
domains~\cite[Lemma 3.4]{AlessandriniMagnanini}.

\begin{figure}
  \centering
  \includegraphics[width=6cm]{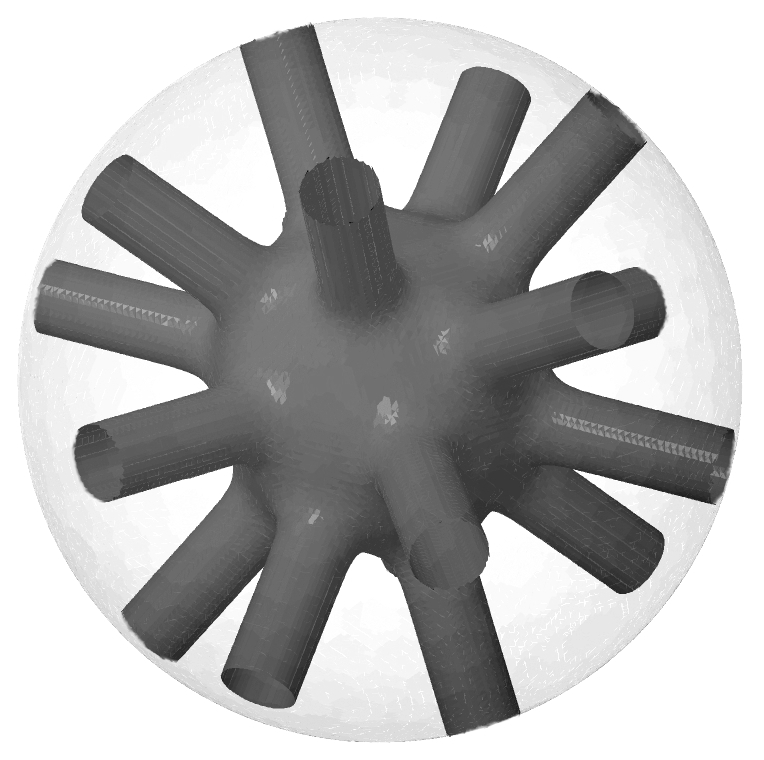}
  \caption{A surface inside a ball creating only two connected components in the interior and a large number of connected components on the boundary sphere.}
  \label{figure:urchin}
\end{figure}
In higher dimensions, the number of interior nodal domains does not
control the number of boundary nodal domains (see Figure~\ref{figure:urchin}), 
and therefore 
new ideas are needed to tackle Open Problem \ref{opennodal}.  However,
there are indications that a Courant-type  (i.e. $O(n)$) bound should
hold in this case as well. For instance, this is the case for
cylinders and Euclidean balls (see Examples~\ref{balls} and~\ref{cyl}). 

\subsection{Geometry of the nodal sets} The nodal sets of Steklov eigenfunctions, both interior and boundary, remain largely unexplored.
The basic property of the nodal sets of Laplace--Beltrami
eigenfunctions is their density on the scale of $1/{\sqrt{\lambda}}$,
where $\lambda$ is the eigenvalue  (cf.~\cite{Zelditch2}, see also Figure~\ref{figure:nodalellipse}).
This means that for any manifold
$\Omega$, there exists a constant $C$ such that for any eigenvalue
$\lambda$ large enough, the corresponding eigenfunction $\phi_\lambda$
has a zero in any geodesic ball of radius $C/\sqrt{\lambda}$.  This
motivates the following questions (see also Figure \ref{figure:nodalellipse}): 
\begin{open}
  (i) Are the nodal sets of Steklov eigenfunctions on a Riemannian
  manifold $\Omega$ dense on the scale $1/\sigma$ in $\Omega$? 
  (ii) Are the nodal sets of the Dirichlet-to-Neumann eigenfunctions
  dense on the scale $1/\sigma$ in $M=\partial \Omega$?
\end{open}

For smooth simply--connected planar domains, a positive
answer to question (ii) follows from the work of Shamma~\cite{shamma} on
asymptotic behaviour of Steklov eigenfunctions. On the other hand, the explicit representation of eigenfunctions on rectangles
implies that there exist eigenfunctions of arbitrary high order which have zeros only on one  pair of parallel sides. Therefore,  a positive answer to (ii)
may possibly  hold only under some regularity assumptions on the boundary.

\begin{figure}
  \centering
  \includegraphics[width=8cm]{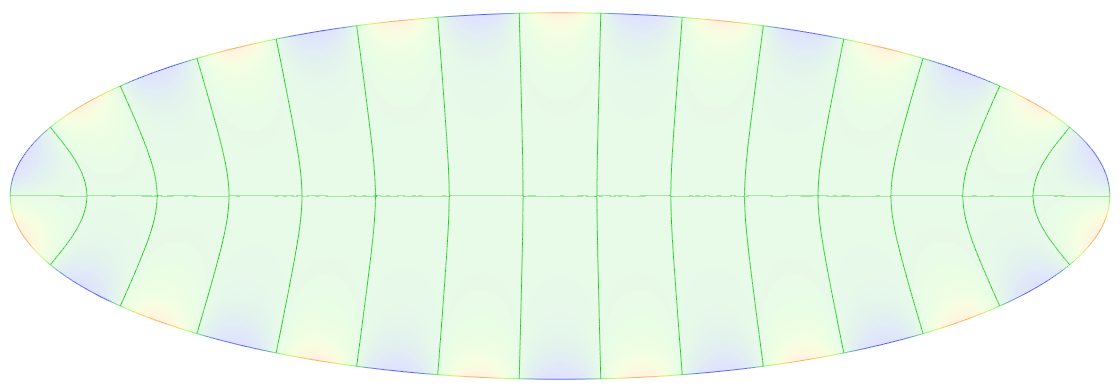}
  \caption{The nodal lines of the 30th eigenfunction on an ellipse.}
  \label{figure:nodalellipse}
\end{figure}

Another fundamental  problem in nodal geometry is to estimate the size of the nodal set. It was conjectured by S.-T. Yau that for any Riemannian manifold of dimension $n$, 
$$C_1\sqrt{\lambda} \le \mathcal{H}_{n-1}(\mathcal{N}(\phi_\lambda)) \le C_2\sqrt{\lambda},$$
where $\mathcal{H}_{n-1}(\mathcal{N}(\phi_\lambda))$ denotes the $n-1$-dimensional Hausdorff measure of the nodal set $\mathcal{N}(\phi_\lambda)$
of a Laplace-Beltrami eigenfunction $\phi_\lambda$, and the constants $C_1, C_2$ depend only on the geometry of the manifold. Similar questions can be asked in the Steklov setting:
\begin{open} 
\label{size}
Let $\Omega$ be an $n$-dimensional Riemannian manifold with boundary $M$. Let $\phi_\sigma$ be an eigenfunction of the Steklov problem on $\Omega$ corresponding to the eigenvalue $\sigma$ and let $u_\sigma=\phi_\sigma|_M$ be the corresponding eigenfunction of the Dirichlet-to-Neumann operator on $M$. Show that

\smallskip

(i)  $C_1\sigma \le \mathcal{H}_{n-1}(\mathcal{N}(\phi_\sigma)) \le C_2\sigma,$

\smallskip

(ii) $C_1' \sigma \le \mathcal{H}_{n-2}(\mathcal{N}(u_\sigma)) \le C_2'\sigma,$

\smallskip

\noindent where the constants $C_1,C_2, C_1', C_2'$ depend only on the manifold.
\end{open}

Some partial results on this problem are known.  In particular, the upper bound  in (ii) was
conjectured by~\cite{BellovaLin} and proved in~\cite{Zelditch2}
for real analytic manifolds with real analytic boundary.  A lower bound on the size of the nodal set $\mathcal{N}(u_\sigma)$   for smooth Riemannian manifolds   (though weaker than the one conjectured in (ii) in dimensions $\ge 3$)
was recently obtained in~\cite{WangZhu} using an adaptation of the approach of~\cite{SoggeZelditch} to nonlocal operators.

The upper bound in (i) is related to the question of  estimating the size of the zero set of a harmonic function in terms of its frequency (see \cite{HanLin}).
In \cite{PoltSherToth}, this approach is combined with the methods of  potential theory and complex analysis in order to obtain both upper and lower bounds in (i) for simply--connected analytic surfaces.
 Let us also note that the Steklov eigenfunctions decay rapidly away
 from the boundary~\cite{HislopLutzer}, and therefore the problem of
 understanding the properties of the nodal set in the interior is
 somewhat analogous to the study of the zero sets of Schr\"odinger
 eigenfunctions in the ``forbidden regions'' (see~\cite{HaZelZhou}).

\subsection{Multiplicity bounds for Steklov eigenvalues}
In two dimensions, the estimate on the number of nodal domains allows
to control the eigenvalue multiplicities (see \cite{Besson, Cheng}).
The argument roughly goes as follows:
if the multiplicity of an eigenvalue is high, one can construct a
corresponding  eigenfunction with a high enough vanishing order at a
certain point of a surface.  In the neighbourhood of this point the
eigenfunction looks like a harmonic polynomial, and therefore the
vanishing order together with the topology of a surface yield a lower
bound on the number of nodal domains. To avoid a contradiction with
Courant's theorem, one deduces a bound on the vanishing order, and
hence on the multiplicity. 

This general scheme was originally applied to Laplace-Beltrami
eigenvalues, but it can be also adapted to prove multiplicity bounds for Steklov eigenvalues.  Interestingly enough, one can obtain estimates of two kinds. 
Recall that the Euler characteristic $\chi$ of an orientable surface of genus
$\gamma$ with $l$ boundary components equals  $2-2\gamma-l$, and of a non-orientable one is equal to $2-\gamma-l$.
Putting together the results of~\cite{KarKoPo,
  jammes2, jammes3, fraschoen2} we get the following bounds:
\begin{theorem}
  Let $\Sigma$ be a compact surface of Euler characteristic $\chi$
  with $l$ boundary components. Then the multiplicity $m_k(\Sigma)$
  for any $k \ge 1$ satisfies the following inequalities: 
  \begin{equation}
    \label{mult1}
    m_k(\Sigma) \le 2k-2\chi-2l+5,
  \end{equation}
  \begin{equation}
    \label{mult2}
    m_k(\Sigma) \le k-2\chi+3.
  \end{equation}
\end{theorem}
Note that the right-hand side of \eqref{mult1} depends only on  the index of the eigenvalue $k$ and on the genus $\gamma$ of the surface, while
the right-hand side of \eqref{mult2} depends also on the number of boundary components.
Inequality \eqref{mult2} in this form was proved in \cite{jammes3}. In
particular, it is sharp for the first eigenvalue of the  disk
($\chi=2$, $l=1$, the maximal multiplicity is two) and of the M\"obius
band ($\chi=0$, $l=1$, the maximal multiplicity is four). Inequality \ref{mult1} is sharp for the annulus ($\chi=0$, $l=2$, the
maximal multiplicity is three and attained  by the critical catenoid,
see Theorem \ref{thm:fraschoencritcat}).

While these bounds are sharp in some cases, they are far from optimal
for large~$k$. In fact, the following result is an immediate corollary
of Theorem \ref{main:GPPS}.
\begin{corollary} \cite{GPPS}
  For any smooth compact Riemannian surface $\Omega$ with $l$ boundary
  components, there is a constant $N$ depending on the metric on
  $\Omega$ such that for $j>N$, the multiplicity of $\sigma_j$ is at
  most $2l$.
\end{corollary}

\begin{remark}
  The multiplicity of the first nonzero eigenvalue $\sigma_1$ has been
  linked to the relative chromatic number of the corresponding surface with boundary  in~\cite{jammes3}.
\end{remark}

For manifolds of dimension $n\ge 3$, no general multiplicity bounds
for Steklov eigenvalues are available. Moreover, given a Riemannian
manifold $\Omega$ of dimension $n\ge 3$ and  any non-decreasing
sequence of $N$ positive numbers, one can find a Riemannian metric $
g$ in  a given conformal class, such that this sequence  coincides
with the first $n$ nonzero Steklov eigenvalues of $(M,g)$
\cite{jammes1}.
\begin{theorem}\label{thm:prescriptionJammes}
  Let $\Omega$ be a compact manifold with boundary. Let $n$ be a
  positive integer and let
  $0=s_0< s_1\leq\cdots\leq s_n$ be a finite
  sequence. Then there exists a Riemannian metric $g$ on $\Omega$ such
  that $\sigma_j=s_j$ for $j=0,\cdots,n$.
\end{theorem}
 For Laplace-Beltrami eigenvalues, a similar result was
obtained in \cite{CdV87}. It is plausible that multiplicity bounds for
Steklov eigenvalues in higher dimensions could be obtained under
certain geometric assumptions, such as curvature constraints.
\subsection*{Acknowledgements}
The authors would like to thank Brian Davies for inviting them to
write this survey. The project started in 2012 at the conference on 
\emph{Geometric Aspects of Spectral Theory} 
at the Mathematical Research Institute in Oberwolfach,
and its hospitality is  greatly appreciated.
We are grateful to Mikhail Karpukhin and David Sher for helpful
remarks on the preliminary version of the paper. We are also thankful to
Dorin Bucur, Fedor Nazarov, Alexander Strohmaier and John Toth for useful discussions, 
as well as to Bartek Siudeja for letting us use his FEniCS
eigenvalues computation code.

\bibliographystyle{plain}
\bibliography{biblioJST1}
\end{document}